\documentclass[leqno]{article} 

\usepackage[frenchb,english]{babel}
\usepackage[T1]{fontenc}

\usepackage{amsmath}
\usepackage{amssymb}
\usepackage{amsfonts}
\usepackage{enumerate}
\usepackage{vmargin}
\usepackage[all]{xy}
\usepackage{mathrsfs}
\usepackage{mathtools}
\usepackage{hyperref}

\setmarginsrb{3cm}{3.25cm}{2.5cm}{2cm}{0cm}{0cm}{1.5cm}{3cm}

\newcommand{\N}{\ensuremath{\mathcal{N}}}

\newcommand{\C}{\ensuremath{\mathbb{C}}}

\newcommand{\VN}{{\rm VN}}

\renewcommand{\leq}{\ensuremath{\leqslant}}
\renewcommand{\geq}{\ensuremath{\geqslant}}
\newcommand{\n}{\noindent}
\newcommand{\qed}{\hfill \vrule height6pt  width6pt depth0pt}
\newcommand{\norm}[1]{ \| #1  \|}
\newcommand{\bnorm}[1]{ \big\| #1  \big\|}
\newcommand{\Bnorm}[1]{ \Big\| #1  \Big\|}
\newcommand{\bgnorm}[1]{ \bigg\| #1  \bigg\|}
\newcommand{\Bgnorm}[1]{ \Bigg\| #1  \Bigg\|}

\newcommand{\ot}{\otimes}
\newcommand{\ovl}{\overline}
\newcommand{\otvn}{\ovl\ot}
\newcommand{\ul}{\mathcal{U}}
\newcommand{\Rad}{{\rm Rad}}

\newcommand{\co}{\colon}
\newcommand{\OUMD}{{\rm OUMD}}
\newcommand{\UMD}{{\rm UMD}}
\newcommand{\QWEP}{{\rm QWEP}}
\newcommand{\OK}{{\rm OK}}
\newcommand{\K}{{\rm K}}

\newtheorem{thm}{Theorem}[section]
\newtheorem{defi}[thm]{Definition}
\newtheorem{prop}[thm]{Proposition}

\newtheorem{lemma}[thm]{Lemma}

\newtheorem{remark}[thm]{Remark}

\newenvironment{proof}[1][]{\noindent {\it Proof #1} : }{\hbox{~}\qed
\smallskip
}

\numberwithin{equation}{section}

\begin{document}
\selectlanguage{english}
\title{\bfseries{Analytic semigroups on vector valued noncommutative $L^p$-spaces}}
\date{}
\author{\bfseries{C\'edric Arhancet}}

\maketitle

\begin{abstract}
We give sufficient conditions on an operator space $E$ and on a
semigroup of operators on a von Neumann algebra $M$ to obtain a
bounded analytic or a $R$-analytic semigroup $(T_t \ot Id_E)_{t \geq
0}$ on the vector valued noncommutative $L^p$-space $L^p(M,E)$.
Moreover, we give applications to the $H^\infty(\Sigma_\theta)$
functional calculus of the generators of these semigroups,
generalizing some earlier work of M. Junge, C. Le Merdy and Q. Xu.
\bigskip
\end{abstract}


\makeatletter
 \renewcommand{\@makefntext}[1]{#1}
 \makeatother
 \footnotetext{
 The author is supported by the research program ANR 2011 BS01 008 01. 
}

\section{Introduction}
It is shown in \cite{JMX} that whenever $(T_t)_{t\geq 0}$ is a
noncommutative diffusion semigroup on a von Neumann algebra $M$
equipped with a faithful normal state such that each $T_t$ satisfies
the Rota dilation property, then the negative generator of its
$L^p$-realization ($1< p < \infty$) admits a bounded
$H^\infty(\Sigma_\theta)$ functional calculus for some $0< \theta <
\frac{\pi}{2}$ where $\Sigma_{\theta} = \{ z \in \mathbb{C}^*: |\arg
z| < \theta \}$ is the open sector of angle $2\theta$ around the
positive real axis $(0,+\infty)$. Our first principal result is an
extension of this theorem to the vector valued case. We use a
different approach using $R$-analyticity instead of square
functions.

In order to describe our result, we need several definitions.
\begin{defi}\label{defmarkov}
Let $(M,\phi)$ and $(N,\psi)$ be von Neumann algebras equipped with
normal faithful states $\phi$ and $\psi$ respectively. A linear map
$T\colon M \rightarrow N$ is called a $(\phi, \psi)$-Markov map if
\begin{enumerate}
\item [$(1)$] $T$ is completely positive
\item [$(2)$] $T$ is unital
\item [$(3)$] $\psi\circ T=\phi$
\item [$(4)$] $T\circ \sigma_t^{\phi}=\sigma_t^\psi\circ T$, for all $t\in \mathbb{R}$, where $(\sigma_t^{\phi})_{t\in \mathbb{R}}$ and $(\sigma_t^{\psi})_{t\in \mathbb{R}}$ denote the automorphism groups of the states $\phi$ and $\psi$ respectively.
\end{enumerate}
In particular, when $(M,\phi)=(N,\psi)$, we say that $T$ is a
$\phi$-Markov map.
\end{defi}
A linear map $T \co M \to N$ satisfying conditions
$(1)-(3)$ above is normal. If, moreover, condition $(4)$ is
satisfied, then it is known that there exists a unique completely
positive, unital map $T^*\co N \to M$ such that
\begin{equation*}
\phi\big(T^*(y)x\big) =\psi\big(y T(x)\big), \qquad x\in M,\ y\in N.
\end{equation*}
The next definition is a variation of the one of \cite{AnD} (see
also \cite{Ric} and \cite{HaM}).
\begin{defi}\label{Def QWEP-factorisable}
A $(\phi, \psi)$-Markov map $T\colon M\rightarrow N$ is called
$\QWEP$-factorizable if there exist a von Neumann algebra $P$ with
$\QWEP$ equipped with a faithful normal state $\chi$\,, and
$*$-monomorphisms $J_0\colon N\rightarrow P$ and $J_1 \co M
\to P$ such that $J_0$ is $(\phi, \chi)$-Markov and $J_1$ is
$(\psi, \chi)$-Markov, satisfying, moreover, $T=J_0^*\circ J_1$\,.
We say that $T$ is hyper-factorizable if the same property is true
with a hyperfinite von Neumann algebra $P$.
\end{defi}

Now, we introduce the following definition (compare \cite[Definition
4.1]{HaM} and \cite[Property 4.10]{Arh})

\begin{defi}
\label{Def QWEP dilatable} Let $M$ be a von Neumann algebra equipped
with a normal faithful state $\phi$. Let $(T_t)_{t\geq 0}$ be a
$w^*$-continuous semigroup of $\phi$-Markov maps on $M$. We say that
the semigroup is $\QWEP$-dilatable if there exist a von Neumann
algebra $N$ with $\QWEP$ equipped with a normal faithful state
$\psi$, a w*-continuous group $(U_t)_{t \in \mathbb{R}}$ of
$*$-automorphisms of $N$, a $*$-monomorphism $J \co M \to N$
such that each $U_t$ is $\phi$-Markov and $J$ is $(\phi,
\psi)$-Markov satisfying
\begin{eqnarray*}
T_t=\mathbb{E}\circ U_t \circ J,\qquad t \geq 0,
\end{eqnarray*}
where $\mathbb{E}=J^* \co N \to M$ is the canonical faithful
normal conditional expectation preserving the states associated with
$J$.
\end{defi}
Let $C^*(\mathbb{F}_{\infty})$ be the full group $C^*$-algebra of
the free group $\mathbb{F}_{\infty}$. We say that an operator space
$E$ is locally-$C^*(\mathbb{F}_{\infty})$ if
$$
d_f(E) = \sup_{F \subset E, \text{finite dimensional}} \inf \big\{
d_{cb}(F,G) : G \subset C^*(\mathbb{F}_{\infty}) \big\} < \infty.
$$
This property is stable under duality and complex interpolation. All
natural examples satisfy $d_f(E)=1$ (see \cite[Chapter 21]{Pis6} and
\cite{Har} for more information on this class of operator spaces).
If $M$ be a von Neumann algebra with QWEP equipped with a normal
faithful state and if $E$ is locally-$C^*(\mathbb{F}_{\infty})$,
then the vector valued non commutative $L^p$-space $L^p(M,E)$ is
well-defined and generalize the classical construction for
hyperfinite von Neumann algebras (See section 2 for more
information).

For any index set $I$, we denote by $OH(I)$ the associated operator
Hilbert space introduced by G. Pisier, see \cite{Pis7} and
\cite{Pis6} for the details. Recall that $\OUMD_p$ is the operator
space analogue of the \textit{Unconditional Martingale Differences}
(UMD) property of Banach spaces, see section 2 for more
information. We also use a similar property $\OUMD'_p$ for QWEP von
Neumann algebras. The definition is given in section 2. Our main
result is the following theorem.

\begin{thm}\label{Th calcul fonctionnel complètement borné}
Let $M$ be a von Neumann algebra with $\QWEP$ equipped with a normal
faithful state. Let $(T_t)_{t\geq 0}$ be a $\QWEP$-dilatable
$w^*$-continuous semigroup of operators on $M$. Suppose $1< p,q
<\infty$ and $0<\alpha<1$. Let $E$ be an operator space such that
$E=\big(OH(I),F\big)_\alpha$ for some index set $I$ and for some
operator space $F$ with
$\frac{1}{p}=\frac{1-\alpha}{2}+\frac{\alpha}{q}$. Assume one of the
following assertions.
\begin{enumerate}
  \item Each $T_t$ is hyper-factorizable and $F$ is $\OUMD_{q}$.
  \item Each $T_t$ is $\QWEP$-factorizable and $F$ is $\OUMD_{q}'$.
\end{enumerate}
We let $-A_p$ be the generator of the strongly continuous semigroup
$(T_t\ot Id_{E})_{t\geq 0}$ on $L^p(M,E)$. Then for some
$0<\theta<\frac{\pi}{2}$, the operator $A_p$ has a completely
bounded $H^{\infty}(\Sigma_\theta)$ functional calculus.
\end{thm}

This result can be used with the noncommutative Poisson semigroup or
the $q$-Ornstein-Uhlenbeck semigroup and by example $E=OH(I)$ for
any index set $I$. A version of this result for semigroups of Schur
multipliers is also given.

A famous theorem of G. Pisier \cite{Pis4} says that a Banach space
$X$ is K-convex (i.e. does not contain uniformly $\ell^1_n$'s) if
and only if the vectorial Rademacher projection $P \ot Id_X$ is
bounded on the Bochner space $L^2(\Omega,X)$ where $\Omega$ is a
probability space. In this proof, he showed and used the fact that
if $X$ is K-convex then any $w^*$-continuous semigroup
$(T_t)_{t \geq 0}$ of positive unital selfadjoint Fourier multipliers
on a locally compact abelian group $G$ induces a bounded analytic
semigroup $(T_t \ot Id_X)_{t \geq 0}$ on the Bochner space $L^p(G,X)$
where $1< p <\infty$. Moreover, he showed a similar result for
general $w^*$-continuous semigroups of positive unital contractions
on a measure space if $X$ does not contain, for some $\lambda>1$,
any subspace $\lambda$-isomorphic to $\ell^1_2$. We give noncommutative analogues of these results. There are crucial steps in the proof of our Theorem \ref{Th calcul fonctionnel complètement
borné}.

We say that an operator space $E$ is OK-convex if the vector
valued Schatten space $S^2(E)$ is K-convex. This notion was
introduced in \cite{JuP}. Using the preservation of the
K-convexity under complex interpolation (see \cite{Pis1}), it is
easy to see that it is equivalent to the K-convexity of the Banach
space $S^p(E)$ for any (all) $1<p<\infty$.

Our second principal result is the following theorem:

\begin{thm}
\label{Th Analyticity Fourier multipliers} Suppose that $G$ is an
amenable discrete group or that $G$ is the free group $\mathbb{F}_n$
with $n$ generators ($1\leq n \leq \infty$). Let $(T_t)_{t \geq 0}$ be
a $w^*$-continuous semigroup of self-adjoint completely positive
unital Fourier multipliers on the group von Neumann algebra $\VN(G)$
preserving the canonical trace. Let $E$ be a $\OK$-convex operator
space. If $G=\mathbb{F}_n$, we suppose that $E$ is locally-$C^*(\mathbb{F}_{\infty})$ with $d_f(E)=1$. Consider $1<p<\infty$. Then $(T_t \ot Id_{E})_{t\geq 0}$ defines a strongly continuous bounded analytic semigroup on the Banach space $L^p\big(\VN(G),E\big)$.
\end{thm}
We will show that this result can be used, by example, in the case
where $E$ is a Schatten space $S^q$ or a commutative $L^q$-space
with $1<q<\infty$.

The next theorem is a $R$-analytic version of Theorem \ref{Th
Analyticity Fourier multipliers} and it is our last principal
result.
\begin{thm}
\label{Th R-Analyticity general semigroups} Let $M$ be a von Neumann
algebra with $\QWEP$ equipped with a normal faithful state. Let
$(T_t)_{t\geq 0}$ a $w^*$-continuous semigroup of operators on $M$.
Suppose $1< p,q <\infty$ and $0<\alpha<1$. Let $E$ be an operator
space such that $E=\big(OH(I),F\big)_\alpha$ for some index set $I$
and for some operator space $F$ with $\frac{1}{p}=\frac{1-\alpha}{2}+\frac{\alpha}{q}$. Assume one of the following assertions.
\begin{enumerate}
  \item Each $T_t$ is hyper-factorizable and $F$ is $\OUMD_{q}$.
  \item Each $T_t$ is $\QWEP$-factorizable and $F$ is $\OUMD_{q}'$.
\end{enumerate}
Then $(T_t \ot Id_{E})_{t \geq 0}$ defines a strongly continuous $R$-analytic semigroup on the Banach space $L^p(M,E)$.
\end{thm}
Finally, we also give a version of these two theorems for semigroups of Schur multipliers.


The paper is organized as follows. Section 2 gives a brief
presentation of vector valued noncommutative $L^p$-spaces and we
introduce some notions which are relevant to our paper. The next
section 3 contains a proof of Theorem \ref{Th Analyticity Fourier
multipliers}. Section 4 is devoted to prove Theorem \ref{Th
R-Analyticity general semigroups}. In section 5, we give
applications to functional calculus. In particular, we prove Theorem
\ref{Th calcul fonctionnel complètement borné}. Finally, we present
some natural examples to which the results of this paper can be
applied.

\section{Preliminaries}
The readers are referred to \cite{ER}, \cite{Pau} and \cite{Pis6}
for details on operator spaces and completely bounded maps and to
the survey \cite{PiX} for noncommutative $L^p$-spaces.

The theory of vector valued noncommutative $L^p$-spaces was
initiated by G. Pisier \cite{Pis5} for the case where the underlying
von Neumann algebra is hyperfinite and equipped with a faithful
normal semifinite trace. Suppose $1 \leq p < \infty$. For an
operator space $E$ and a hyperfinite von Neumann algebra $M$ equipped with a faithful normal semifinite trace, we define by complex interpolation
\begin{equation}
\label{Def vector valued Lp non com}
L^p(M,E)=\big(M \ot_{\min} E, L^1(M)\widehat{\ot}
E\big)_{\frac{1}{p}},
\end{equation}
where $\ot_{\min}$ and $\widehat{\otimes}$ denote the injective and
the projective tensor product of operator spaces. In \cite{Jun1} and
\cite{Jun2}, M. Junge extended this theory to the case where the
underlying von Neumann algebra satisfies QWEP using the following
characterization of QWEP von Neumann algebras. It is unknown
whether every von Neumann algebra has this property. See the survey
\cite{Oza} for more information on this notion.

\begin{prop}
\label{prop Carac des QWEP}
A von Neumann algebra $M$ is $\QWEP$ if and only if there exist an
index set $I$, a free ultrafilter $\ul$ on an index set $L$, a
normal $*$-monomorphism
$$
\pi \co M \rightarrow B\big(\ell^2_I\big)^{\ul}
$$
and a normal conditional expectation
$$
\mathbb{E} \co B\big(\ell^2_I\big)^{\ul} \rightarrow \pi(M)
$$
where $B\big(\ell^2_I\big)^{\ul}$ denote the ultrapower of $B\big(\ell^2_I\big)$ associated with $\ul$.
\end{prop}
Note that, in general, $\mathbb{E}$ is not faithful.
%
%

Now, we introduce the vector valued noncommutative $L^p$-spaces
associated to a von Neumann algebra with QWEP equipped with a normal faithful state. Let $M$, $\beta=(\pi,\mathbb{E},\ul,I)$ be as in Proposition \ref{prop Carac des QWEP}. Suppose $1 < p < \infty$.
Let $E$ be an operator space. Recall that the vector valued noncommutative $L^p$-space $L^p(M,\beta,E)$ is a closed subspace of the ultrapower $S^p_I(E)^\ul$ which depends on the choice of $\beta$ (see \cite{Jun2} for a precise definition).


However, when $E$ is locally-$C^*(\mathbb{F}_{\infty})$, the
$L^p(M,\beta, E)$'s are all equivalent allowing the constant
$d_f(E)^2$. Thus, we can still say that $L^p(M,\beta,E)$ does not
depend on the choice of $\beta$, and thus we will use the most
convenient notation $L^p(M,E)$ instead of $L^p(M,\beta,E)$.

Note the following vector valued extension property of completely
positive maps between noncommutative $L^p$-spaces, see \cite{Pis2} and \cite{Jun2} . 
\begin{prop}
\label{prop-tensorisation of CP maps}
 Suppose $1 < p < \infty$.
\begin{enumerate}
  \item Let $M$ and $N$ be hyperfinite von Neumann algebras equipped with normal semifinite faithful traces and let $E$ be an operator space. Let $T \co M \to N$ be a trace preserving unital normal completely positive map. Then the operator $T \ot Id_E$ induces a bounded operator from $L^p(M,E)$ into $L^p(N,E)$ and we have
$$
\norm{T \ot Id_E}_{L^p(M,E) \to L^p(N,E)} \leq \norm{T}_{L^p(M) \to
L^p(N)}.
$$
  \item Let $M$ and $N$ be von Neumann algebras with $\QWEP$ equipped with
normal faithful states $\phi$ and $\psi$ respectively and let $E$ be
a locally-$C^*(\mathbb{F}_{\infty})$ operator space.  Let $T \co M \to N$ be a $(\phi, \psi)$-Markov map. Then the operator $T \ot Id_E$ induces a bounded operator from $L^p(M,E)$ into $L^p(N,E)$ and we have
$$
\norm{T \ot Id_E}_{L^p(M,E) \to L^p(N,E)} \leq d_{f}(E)
\norm{T}_{L^p(M) \to L^p(N)}.
$$
\end{enumerate}
\end{prop}

%

%

Suppose that $G$ is a discrete group. We denote by $e_G$ the neutral
element of $G$. We denote by $\lambda_g \co \ell^2_G \to \ell^2_G$ the unitary operator of left translation by $g$ and $\VN(G)$ the von Neumann algebra of $G$ spanned by the $\lambda_g$ where $g \in G$. It is an finite algebra with normalized trace given by
$$
\tau_{G}(x)=\big\langle\epsilon_{e_G},x(\epsilon_{e_G})\big\rangle_{\ell^2_G}
$$
where $(\epsilon_g)_{g \in G}$ is the canonical basis of $\ell^2_G$ and $x \in \VN(G)$. For any $g \in G$, note that we have
\begin{equation}
\label{tracelambdag}
    \tau_{G}(\lambda_g)=\delta_{g,e}.
\end{equation}
Recall that the von Neumann algebra $\VN(G)$ is hyperfinite if and only if $G$ is amenable \cite[Theorem 3.8.2]{SS}. A Fourier multiplier is a normal linear map $T \co \VN(G) \to \VN(G)$ such that there exists a function $\varphi \co G \to \C$ such that for any $g \in G$ we have $T(\lambda_g)=\varphi_g \lambda_g$. In this case, we denote $T$ by
$$
\begin{array}{cccc}
   M_\varphi \co   &    \VN(G)      &  \longrightarrow   & \VN(G)   \\
          &    \lambda_g  &  \longmapsto       & \varphi_g\lambda_g.   \\
\end{array}
$$

Let $M$ be a von Neumann algebra equipped with a semifinite normal
faithful trace $\tau$. Suppose that $T \co M \to M$ is a normal
contraction. We say that $T$ is selfadjoint if for all $x,y \in
M\cap L^1(M)$ we have
$$
\tau\big(T(x)y^*\big)=\tau\big(x(T(y))^*\big).
$$
In this case, it is not hard to show that the restriction $T|M \cap
L^1(M)$ extends to a contraction $T \co L^1(M) \to L^1(M)$. By
complex interpolation, for any $ 1\leq p\leq \infty$, we obtain a
contractive map $T \co L^p(M) \to L^p(M)$. Moreover, the
operator $T \co L^2(M) \to L^2(M)$ is selfadjoint. If $T\colon
M \to M$ is actually a normal selfadjoint complete contraction,
it is easy to see that the map $T \co L^p(M) \to L^p(M)$ is
completely contractive for any $ 1 \leq p \leq \infty$. It is not
difficult to show that a contractive Fourier multiplier $M_\varphi
\co \VN(G) \to \VN(G)$ is selfadjoint if and only if $\varphi
\co G \to \C$ is a real function. Finally, one can prove that
a contractive Schur multiplier $M_A \co B\big(\ell^2_I\big) \to
B\big(\ell^2_I\big)$ associated with a matrix $A$ 
is selfadjoint if and only if all entries of $A$ are real.

Now, we introduce the operator space version of the Banach space
property UMD. Let $E$ be an operator space and $1< p< \infty$. We
say that $E$ is $\OUMD_p$ if there exists a positive constant $C$
such that for any positive integer $n$, any choice of signs
$\varepsilon_{k}=\pm 1$ and any martingale difference sequence
$(dx_k)_{k=1}^n \subset L^p(M,E)$ relative to a filtration
$(M_k)_{k\geq 1}$ of a hyperfinite von Neumann algebra $M$ equipped
with a normal faithful finite trace we have
\begin{equation}
\label{defumd} 
\Bgnorm{\sum_{k=1}^n\varepsilon_k dx_k}_{L^p(M, E)}
\leq C \Bgnorm{\sum_{k=1}^n dx_k}_{L^p(M,E)}.
\end{equation}
See \cite{Mus} and \cite{Qiu} for more information on this property.
We also need a variant of this property for QWEP von
Neumann algebras.

\begin{defi}
Suppose $1< p< \infty$. Let $E$ be a
locally-$C^*(\mathbb{F}_{\infty})$ operator space with $d_f(E)=1$.
We say that $E$ is $\OUMD_p'$ if there exists a positive constant $C$
such that for any positive integer $n$, any choice of signs
$\varepsilon_{k}=\pm 1$ and any martingale difference sequence
$(dx_k)_{k=1}^n \subset L^p(M,E)$ relative to a filtration
$(M_k)_{k\geq 1}$ of a $\QWEP$ von Neumann algebra $M$ equipped with a
normal faithful state we have the inequality (\ref{defumd}).
\end{defi}
Suppose $1<p<\infty$. Any $\OUMD_p'$-operator space $E$ is $\OUMD_p$.
Let $1<q,r<\infty$ and $0<\theta<1$ be such that $\frac{1}{p}=\frac{1-\alpha}{q}+\frac{\alpha}{r}$. If $(E_0,E_1)$ is a compatible couple of operator spaces, where $E_0$ is $\OUMD_q$ and $E_1$ is $\OUMD_r$, then the complex interpolation operator space $E_\theta=(E_0,E_1)_\alpha$ is $\OUMD_p$. For any index set $I$ and any $1<p<\infty$, the operator Hilbert space $OH(I)$ is $\OUMD_p$. If $E$ is $\OUMD_p$ then the Banach space $S^p(E)$ is UMD (hence
K-convex). It is easy to see that the same properties are valid
for $\OUMD_p'$ operator spaces (with the same proofs).




\section{Analyticity}

Let $X$ be a Banach space. A strongly continuous semigroup
$(T_t)_{t\geq 0}$ is called bounded analytic if there exist
$0<\theta<\frac{\pi}{2}$ and a bounded holomorphic extension
$$
\begin{array}{cccc}
    &  \Sigma_\theta   &  \longrightarrow   &  B(X)  \\
    &  z   &  \longmapsto       &  T_z.  \\
\end{array}
$$
See \cite{KW} and \cite{Haa} for more information on this notion.

We need the following theorem which is a corollary \cite[Lemma
4]{Pis4} of a result of Beurling \cite{Beu} (see also \cite{Fac} and \cite{Hin}).
\begin{thm}
\label{Théorème de Beurling} 
Let $X$ be a Banach space. Let $(T_t)_{t\geq 0}$ be a strongly continuous semigroup of contractions on $X$. Suppose that there exists some integer $n \geq 1$ such that for any $t> 0$
\begin{equation*}\label{}
\bnorm{(Id_{X}-T_t)^n}_{X \to X} < 2^{n}.
\end{equation*}
Then the semigroup $(T_t)_{t \geq 0}$ is bounded analytic.
\end{thm}

Moreover, we will use the following lemma \cite[Lemma 1.5]{Pis4}.
\begin{lemma}
\label{Lemma of Pisier on projections} Let $n\geq 1$ be an integer.
Suppose that $X$ is a real Banach space, such that, for some $\lambda>1$, $X$ does not contain any subspace $\lambda$-isomorphic to $\ell^1_{n+1}$. Then there exists a real number $0<\rho<2$ such that if $P_1,\ldots, P_n$ is any finite collection of mutually commuting norm one projections on $X$, then
\begin{equation*}
\Bgnorm{\prod_{1 \leq k \leq n}(Id_{X}-P_k)}_{X \to X} \leq \rho^n.
\end{equation*}
\end{lemma}
Note that it is known that a complex Banach space contains
$\ell_n^1$'s uniformly if and only if the underlying real Banach space has the same property. 

Finally, we recall the next result (see \cite[Lemma 1]{DH} and
\cite[Theorem 2.14]{Knu} for a complete proof).
\begin{lemma}
\label{Lemma isomorphism and distributions} Let $M$ and $N$ be two
von Neumann algebras equipped with faithful normal finite traces,
and let $(x_i)_{i \in I} \in M$ and $(y_i)_{i \in I} \in N$ be
families of elements of $M$ and $N$ respectively that have the same
$*$-distribution. Then the von Neumann algebras generated
respectively by the $x_i$'s and the $y_i$'s are isomorphic, with a
normal $*$-isomorphism sending $x_i$ on $y_i$ and preserving the
trace.
\end{lemma}

We will use the next result which is a variant of well-known Fell's
absorption principle (see \cite[Proposition 8.1]{Pis6}). This 
proposition is a substitute for the trick of G. Pisier \cite[Lemma
1.6]{Pis4}.
\begin{prop}
\label{Prop Fell absorption} 
Let $G$ be a discrete group and $1\leq
p < \infty$. Let $E$ be an operator space. If $G$ is non-amenable, we assume that $\VN(G)$ has $\QWEP$ and that $E$ is locally-$C^*(\mathbb{F}_{\infty})$ with $d_f(E)=1$. For any function $a \co G \to E$ finitely supported on $G$, we have
\begin{equation}
\label{Fell absorption}
\Bgnorm{\sum_{g \in G} \lambda_g \ot \lambda_g \ot \cdots \ot \lambda_g \ot a_g}_{L^p(\ovl{\ot}_{i=1}^n \VN(G),E)}=\Bgnorm{\sum_{g \in G} \lambda_g \ot a_g}_{L^p(\VN(G),E)}.
\end{equation}
Moreover, for any completely positive unital Fourier multiplier
$M_\varphi \co \VN(G) \to \VN(G)$ preserving the canonical trace and
any positive integer $n$ we have
\begin{equation}\label{Majoration norme multiplicateur}
\bnorm{M_\varphi^n \ot Id_E}_{L^p(\VN(G),E) \to L^p(\VN(G),E)}\leq
\norm{M_\varphi^{\ot n} \ot Id_E}_{L^p(\ovl{\otimes}_{i=1}^n
\VN(G),E) \to L^p(\ovl{\otimes}_{i=1}^n \VN(G),E)}.
\end{equation}
\end{prop}

\begin{proof}
Suppose that $m$ is an integer and that $g_{1},\ldots,g_{m} \in G$.
Let $\eta_1,\ldots,\eta_m\in\{*,1\}$ and
$\varepsilon_1,\ldots,\varepsilon_m\in\{-1,1\}$ the associated sign
(i.e $\varepsilon_i=-1$ if and only if $\eta_i=*$). For any integer
$n$, using (\ref{tracelambdag}), we see that
\begin{align*}
\MoveEqLeft \tau_G^{\ot n}\Big( \big(\lambda_{g_{r_1}}\ot
\cdots\ot\lambda_{g_{r_1}}\big)^{\eta_1}\cdots
\big(\lambda_{g_{r_m}}\ot
\cdots\ot\lambda_{g_{r_m}}\big)^{\eta_m}\Big)\\
&=\tau_G^{\ot n}\Big( \big(\lambda_{g_{r_1}}^{\eta_1}\ot
\cdots\ot\lambda_{g_{r_1}}^{\eta_1}\big)\cdots
\big(\lambda_{g_{r_m}}^{\eta_m}\ot
\cdots\ot\lambda_{g_{r_m}}^{\eta_m}\big)\Big)\\
 &=\tau_G\big(\lambda_{g_{r_1}^{\varepsilon_1}\cdots
g_{r_m}^{\varepsilon_m}}\big)^n\\
&=\tau_G\big(\lambda_{g_{r_1}^{\varepsilon_1}\cdots
g_{r_m}^{\varepsilon_m}}\big)\\
&=\tau_G\Big( (\lambda_{g_{r_1}}^{\eta_1}\cdots
\lambda_{g_{r_m}}^{\eta_m} \Big).
\end{align*}
We infer that the families $\big(\lambda_g \ot \cdots \ot
\lambda_g\big)_{g \in G}$ and $\big(\lambda_g\big)_{g \in G}$ have
the same $*$-distribution with respect to the von Neumann algebras
$\ovl{\ot}_{i=1}^n \VN(G)$ equipped with $\tau_G^{\ot n}$ and $\VN(G)$
equipped with $\tau_G$. We conclude by using Lemma \ref{Lemma
isomorphism and distributions} and Proposition
\ref{prop-tensorisation of CP maps}.

Now, we prove the second part of the proposition. Using (\ref{Fell
absorption}) twice, for any positive integer $n$ and any function $a
\colon G\to E$ finitely supported on $G$, we have
\begin{align*}
\MoveEqLeft  \Bgnorm{(M_\varphi^n \ot Id_E)\bigg(\sum_{g \in G}
\lambda_g \ot a_g\bigg)}_{L^p(\VN(G),E)}
= \Bgnorm{\sum_{g \in G} \varphi_g^n\lambda_g \ot a_g}_{L^p(\VN(G),E)}\\
    &= \Bgnorm{\sum_{g \in G} \varphi_g^n \lambda_g \ot \cdots \ot \lambda_g \ot a_g}_{L^p(\ovl{\ot}_{i=1}^n\VN(G),E)}\\
    &= \Bgnorm{(M_\varphi\ot \cdots \ot M_\varphi \ot Id_E)\bigg(\sum_{g \in G}  \lambda_g \ot \cdots \ot \lambda_g \ot a_g\bigg)}_{L^p(\ovl{\ot}_{i=1}^n \VN(G),E)}\\
    &\leq \Bnorm{M_\varphi \ot \cdots \ot M_\varphi \ot Id_E}_{L^p(\ovl{\ot}_{i=1}^n \VN(G),E) \to L^p(\ovl{\ot}_{i=1}^n\VN(G),E)} \Bgnorm{\sum_{g \in G}  \lambda_g \ot \cdots \ot \lambda_g \ot a_g}_{L^p(\ovl{\ot}_{i=1}^n \VN(G),E)}\\
    &= \Bnorm{M_\varphi \ot\cdots\ot M_\varphi\ot Id_E}_{L^p(\ovl{\otimes}_{i=1}^n \VN(G),E) \to L^p(\ovl{\ot}_{i=1}^n\VN(G),E)} \Bgnorm{\sum_{g \in G} \lambda_g \ot a_g}_{L^p(\VN(G),E)}.
\end{align*}
We deduce (\ref{Majoration norme multiplicateur}).
\end{proof}

\begin{prop}
\label{Prop K-convexity of Lp(M,E)} 
Let $M$ be a von Neumann. Let $E$ be a $\OK$-convex operator space and $1< p < \infty$. Suppose that one of the following assertions is true.
\begin{enumerate}
  \item $M$ is hyperfinite and equipped with a normal faithful semifinite trace.
  \item $M$ has $\QWEP$ and is equipped with a normal faithful state and $E$ is locally-$C^*(\mathbb{F}_{\infty})$.
\end{enumerate}
Then the Banach space $L^p(M,E)$ is $\K$-convex.
\end{prop}

\begin{proof}
We begin with the second case. By definition, the Banach space
$L^p(M,E)$ is a closed subspace of a ultrapower $S^p_I(E)^\ul$ of a
vector valued Schatten space $S^p_I(E)$ for some index set $I$. The
Banach space $S^p(E)$ is K-convex. Hence the Banach space
$S^p_I(E)$ is also K-convex. Recall that K-convexity is a
super-property \cite[page 261]{DJT}, i.e. passes from a Banach space
to all closed subspaces of its ultrapowers. We conclude that the
Banach space $L^p(M,E)$ is K-convex. The first case is similar
using \cite[Theorem 3.4]{Pis5} instead of ultraproducts.
\end{proof}

Now, we are ready to give the proof of Theorem \ref{Th Analyticity
Fourier multipliers}.

\begin{proof}[of Theorem \ref{Th Analyticity Fourier multipliers}]
By \cite[pages 4371]{Ric} and the proof of \cite[Theorem 5.3]{HaM}, for any $t\geq 0$, the operator $T_{t}=(T_{\frac{t}{2}})^2$ admits a Rota dilation (see \cite[Definition 5.1]{HaM} and \cite[page 124]{JMX} for a precise definition)
$$
T_t^k=Q\mathbb{E}_k\pi,\qquad k \geq 1
$$
where $Q \co M \to \VN(G)$ and $\pi \co \VN(G) \to M$ and where we use a crossed product
$$
M=\Gamma_{-1}\big(\ell^{2,T}\ot_2 \ell^2\big)\rtimes_\alpha G
$$
equipped with its canonical trace. We infer that
$$
Id_{L^p(\N(G))}-T_t=Q(Id_{L^p(M)}-\mathbb{E}_1)\pi.
$$
Let $n$ be a positive integer. We deduce that
$$
(Id_{L^p(\VN(G))}-T_t)^{\ot n}=Q^{\ot n}(Id_{L^p(M)}-\mathbb{E}_1)^{\ot n}\pi^{\ot
n}.
$$
For any integer $1 \leq j \leq n$, we let
$$
\Pi_j=Id_{L^p(M)} \ot \cdots \ot Id_{L^p(M)} \ot \mathbb{E}_1 \ot
Id_{L^p(M)} \ot \cdots \ot Id_{L^p(M)}.
$$
Recall that the fermion algebra $\Gamma_{-1}\big(\ell^{2,T} \ot_2 \ell^{2}\big)$ is $*$-isomorphic to the hyperfinite factor of type $\rm{II}_1$. Moreover, if $G$ is amenable then by \cite[Proposition 6.8]{Con}, $M$ is hyperfinite. If $G=\mathbb{F}_n$, by \cite[Proposition 4.8]{Arh}, the von Neumann algebra $M$ has QWEP. By Proposition \ref{prop-tensorisation of CP maps}, we deduce that the $\Pi_j \ot Id_E$'s are well-defined and form a family of mutually commuting contractive projections on $L^p(M,E)$. Moreover, we have
$$
\big((Id_{L^p(M)}-\mathbb{E}_1)\big)^{\ot n} \ot Id_E=\prod_{1 \leq
j \leq n} \Big(Id_{L^p(M^{\ot n},E)}-(\Pi_j\ot Id_E)\Big).
$$
Using (\ref{Majoration norme multiplicateur}), Proposition
\ref{prop-tensorisation of CP maps} and Lemma \ref{Lemma of Pisier
on projections}, we obtain that
\begin{align*}
\MoveEqLeft \Bnorm{\big(Id_{L^p(\VN(G))}-T_t\big)^n\ot
Id_E}_{L^p(\VN(G),E) \to L^p(\VN(G),E)}\\
    & \leq \Bnorm{(Id_{L^p(\VN(G))}-T_t)^{\ot n}\ot Id_E}_{L^p(\ovl{\ot}_{i=1}^n\VN(G),E) \to L^p(\ovl{\ot}_{i=1}^nVN(G),E)}\\
    &= \Bnorm{Q^{\ot n}(Id_{L^p(M)}-\mathbb{E}_1)^{\ot n}\pi^{\ot n} \ot Id_E}_{L^p(\ovl{\ot}_{i=1}^n\VN(G),E) \to L^p(\ovl{\ot}_{i=1}^n\VN(G),E)}\\
    & \leq \Bnorm{(Id_{L^p(M)}-\mathbb{E}_1)^{\ot n}\ot Id_E}_{L^p(M,E) \to L^p(M,E)}\\
    & = \Bgnorm{\prod_{1\leq j\leq n}  \Big(Id_{L^p(M^{\ot n},E)}-(\Pi_j \ot Id_E)\Big)}_{L^p(M^{\ot n},E) \to L^p(M^{\ot n},E)}\\ 
    & \leq \rho^n.
\end{align*}
We conclude by applying Theorem \ref{Théorème de Beurling}. 
\end{proof}

Now, we pass to general semigroups on QWEP von Neumann algebras.
Note that the class of operators spaces considered in the following
theorem is included in the class of OK-convex operator spaces.
\begin{thm}
\label{Th Analyticity general semigroups} Let $M$ be a von Neumann
algebra with $\QWEP$ equipped with a normal faithful state. Let $(T_t)_{t \geq 0}$ be a $w^*$-continuous semigroup of $\QWEP$-factorizable maps on $M$. Suppose $1<p<\infty$. Let $E$ be a locally-$C^*(\mathbb{F}_{\infty})$ operator space with $d_f(E)=1$ such that $S^p(E)$ does not contain, for some $\lambda >1$, any subspace $\lambda$-isomorphic to $\ell_1^2$. Then $(T_t \ot Id_E)_{t\geq 0}$ defines a strongly continuous bounded analytic semigroup on the Banach space $L^p(M,E)$.
\end{thm}

\begin{proof}
Using the fact that each $T_t$ is QWEP-factorizable and
\cite[Theorem 5.3]{HaM}, it is easy to see that each $T_t$ admits a
Rota dilation
$$
T_t^k=Q\mathbb{E}_k\pi,\qquad k \geq 1
$$
where $Q \co N \to M$ and $\pi \co M \to N$ where $N$ is a
QWEP von Neumann algebra. The rest of the proof is similar to the
one of Theorem \ref{Th Analyticity Fourier multipliers}. The only
difference is that we does not use Proposition \ref{Prop Fell
absorption}. We use Lemma \ref{Lemma of Pisier on projections} with
$n=1$. 
\end{proof}

For instance, if $S^p(E)$ (where $1 <p<\infty$) is uniformly convex
then for some $\lambda>1$, $S^p(E)$ does not contain any subspace
$\lambda$-isomorphic to $\ell^1_2$.


Suppose $1<p,q<\infty$. Note that if an operator space $E$ is $\OUMD_q$ then the Banach space $S^q(E)$ is UMD, hence isomorphic to a uniformly convex Banach space. Now, we can write $\frac{1}{p}=\frac{1-\alpha}{q}+\frac{\alpha}{r}$ for some $0 <\alpha <1$ and some $1 < r < \infty$. Then we have $S^p(E)=\big(S^q(E),S^r(E)\big)_\alpha$. Now recall that, by \cite{CwR}, if $(X_0,X_1)$ is a compatible couple of Banach spaces, one of which is uniformly convex, then for any $0<\alpha<1$ the complex interpolation $(X_0,X_1)_\alpha$ is also uniformly convex. We deduce that the Banach space $S^p(E)$ is isomorphic to a uniformly convex space. Hence, we can use in Theorem \ref{Th Analyticity general semigroups} (and also in Theorem \ref{Th Analyticity Fourier multipliers}) any operator space $E$ which is $\OUMD_q$ for \textit{some} $1<q<\infty$ and locally-$C^*(\mathbb{F}_{\infty})$ with $d_f(E)=1$. This large class contains in particular Schatten spaces $S^q$ and commutative $L^q$-spaces for any $1 <q <\infty$.

Finally, we deal with semigroups of Schur multipliers. Note that the
construction of the Rota Dilation for Schur multipliers on finite
dimensional $B(\ell^2_I)$ of \cite{Ric} is actually true for any
index set $I$. Moreover, the von Neumann algebra $\Gamma_{-1}^{e}(\ell^{2,T})$ of \cite{Ric} is hyperfinite. Hence, the von Neumann algebra
$$
M=B\big(\ell^2_I\big) \otvn \Bigg(\ovl{\bigotimes_{n\in \mathbb{N}}}
\big(\Gamma_{-1}^e\big(\ell^{2,T}\big),\tau\big)\Bigg).
$$
of the Rota Dilation constructed with \cite[pages 4369-4370]{Ric} and the proof of \cite[Theorem 5.3]{HaM} is also hyperfinite. Using the above ideas, one can prove the following theorem.
\begin{thm}
\label{Th Analyticity semigroups Schur multipliers} 
Let $(T_t)_{t\geq 0}$ be a $w^*$-continuous semigroup of completely
positive unital selfadjoint Schur multipliers on $B\big(\ell^2_I\big)$. Suppose $1<p<\infty$. Let $E$ be an operator space such that $S^p(E)$ does not contain, for some $\lambda >1$, any subspace $\lambda$-isomorphic to $\ell_1^2$. Then $(T_t \ot Id_E)_{t \geq 0}$ defines a strongly continuous bounded analytic semigroup on the Banach space $S^p_I(E)$.
\end{thm}

\begin{remark}
In \cite[Proposition 5.4]{Arh}, the author gives a concrete description of the semigroups of the above theorem.
\end{remark}

\section{$R$-analyticity}

Let $(\varepsilon_k)_{k \geq 1}$ be a sequence of independent
Rademacher variables on some probability space
$(\Omega,\mathcal{F},\mathbb{P})$. We let $\Rad(X)\subset
L^2(\Omega,X)$ be the closure of ${\rm
Span}\bigl\{\varepsilon_k\otimes x\, :\, k\geq 1,\ x\in X\bigr\}$ in
the Bochner space $L^2(\Omega,X)$. Thus for any finite family
$x_1,\ldots,x_n$ in $X$, we have
$$
\Bgnorm{\sum_{k=1}^{n} \varepsilon_k\otimes x_k}_{\Rad(X)} \,=\,
\Biggr(\int_{\Omega}\bgnorm{\sum_{k=1}^{n} \varepsilon_k(\omega)\,
x_k}_{X}^{2}\,d\omega\,\Biggr)^{\frac{1}{2}}.
$$
We say that a set $F \subset B(X)$ is $R$-bounded provided that there
is a constant $C \geq 0$ such that for any finite families
$T_1,\ldots, T_n$ in $F$ and  $x_1,\ldots,x_n$ in $X$, we have
$$
\Bgnorm{\sum_{k=1}^{n} \varepsilon_k\ot T_k (x_k)}_{\Rad(X)}\,\leq\,
C\, \Bgnorm{\sum_{k=1}^{n} \varepsilon_k \ot x_k}_{\Rad(X)}.
$$
In this case, we let $R(F)$ denote the smallest possible $C$, which
is called the $R$-bound of $F$. Note that any singleton $\{T\}$ is
automatically $R$-bounded with $R\big(\{T\}\big)=\norm{T}_{X \to
X}$. If a set $F$ is $R$-bounded then the strong closure of $F$ is
$R$-bounded (with the same $R$-bound). Recall that if $H$ is a
Hilbert space, a subset $F$ of $B(H)$ is bounded if and only if $F$
is $R$-bounded.

$R$-boundedness was introduced in \cite{BG} and then developed in
the fundamental paper \cite{ClP}. We refer to the latter paper and
to \cite[Section 2]{KW} for a detailed presentation.

The next result is a noncommutative version of the classical result
of Bourgain \cite{Bou} which itself is a vector valued generalization of a result of Stein \cite{Ste}.

\begin{prop}
\label{Prop noncommutative Stein on conditional expectations}Suppose
$1<p<\infty$. Let $E$ be an operator space and let $M$ be a von Neumann algebra. Assume one of the following assertions.
\begin{enumerate}
  \item $M$ is hyperfinite and equipped with a normal faithful semifinite trace and $E$ is
  $\OUMD_p$.
  \item $M$ has $\QWEP$ and is equipped with a normal faithful state and $E$ is
  $\OUMD_p'$.
\end{enumerate}
Consider a increasing (or decreasing) sequence
$\big(\mathbb{E}(\cdot|M_i) \big)_{i \in \mathbb{N}}$ of (canonical)
conditional expectations on some von Neumann subalgebras $M_i$ of
$M$. Then the set of conditional expectation operators  $\big\{\mathbb{E}(\cdot|M_i) \ot Id_{E}\ \colon\ i \in \mathbb{N}\big\}$ is $R$-bounded on $L^p(M,E)$.
\end{prop}

\begin{proof}
We only prove the decreasing case $M_1 \supset M_2 \supset \cdots$
and the hyperfinite case. The proofs of other statements are
similar. First suppose that the trace is finite. Let $n$ be a
positive integer and fix some positive integers $i_1 > i_2 > \cdots
> i_n$. We define the $\sigma$-algebra $\mathcal{F}_k=\sigma(\varepsilon_1,\ldots,\varepsilon_{k})$, $k=1,\ldots,n$ and $\mathcal{F}_{0}=\emptyset$. We define the family $(N_m)_{m=1}^{2n}$ of von Neumann subalgebras of $L^\infty(\Omega) \ot M$ by
$$
N_{2k-1}=L^\infty(\Omega,\mathcal{F}_{k-1},\mathbb{P})\ot M_{i_{k}}
,\qquad k=1,\ldots,n
$$
$$
N_{2k}= L^\infty(\Omega,\mathcal{F}_{k},\mathbb{P}) \ot
M_{i_{k}},\qquad k=1,\ldots,n.
$$
These subalgebras form an increasing sequence $N_1 \subset N_2
\subset \cdots  \subset N_{2n}$. For an element $y \in
L^p\big(L^\infty(\Omega) \ot M,E\big)$, we define the martingale
$(y_m)_{1 \leq m \leq 2n}$ by
$$
y_m=\big(\mathbb{E}(\cdot|N_{m})\ot Id_E\big)(y),\qquad
m=1,\ldots,2n.
$$
Let $(d_m)_{1 \leq m \leq 2n}$ be the associated martingale
difference sequence. Then, by the $\OUMD_p$-property of $E$, we have
\begin{align*}
  \Bgnorm{\sum_{k=1}^{n} d_{2k}}_{L^p(L^\infty(\Omega) \ot M,E)}
    &= \frac{1}{2} \Bgnorm{\sum_{m=1}^{2n} d_{m}+  \sum_{m=1}^{2n} (-1)^m d_{m}}_{L^p(L^\infty(\Omega) \ot M,E)}\\
    &\leq \frac{1}{2} \Bigg(\bgnorm{\sum_{m=1}^{2n} d_{m}}_{L^p(L^\infty(\Omega) \ot M,E)}+  \bgnorm{\sum_{m=1}^{2n} (-1)^m d_{m}}_{L^p(L^\infty(\Omega) \ot M,E)} \Bigg)\\
    & \lesssim \Bgnorm{\sum_{m=1}^{2n} d_{m}}_{L^p(L^\infty(\Omega) \ot M,E)}
\end{align*}
(here and in the sequel of the proof, we use $\lesssim$ to indicate
an inequality up to a constant). Now fix $x_1,\ldots,x_n \in
L^p(M,E)$ and put $y=\sum_{l=1}^{n} \varepsilon_l \ot x_l$. For this
choice of $y$ and any integer $1 \leq k\leq n$, we have
\begin{align*}
   y_{2k-1}
    & = \big(\mathbb{E}(\cdot|N_{2k-1})\ot Id_E\big)(y)\\
    & = \sum_{l=1}^n \big(\mathbb{E}\big(\cdot |L^\infty(\Omega,\mathcal{F}_{k-1},\mathbb{P}) \ot M_{i_k} \big)\ot Id_E\big)(\varepsilon_l \ot x_l)\\
    & = \sum_{l=1}^n \mathbb{E}\big(\varepsilon_l |L^\infty(\Omega,\mathcal{F}_{k-1},\mathbb{P})\big)\ot \big(\mathbb{E}(\cdot | M_{i_k})\ot Id_E\big)(x_l)\\
    & = \sum_{l=1}^{k-1} \varepsilon_l \ot \big(\mathbb{E}(\cdot| M_{i_k}) \ot Id_E\big)(x_l)
\end{align*}
and similarly for any integer $1 \leq k \leq n$ we have
$$
y_{2k}=\sum_{l=1}^{k} \varepsilon_l \ot
\big(\mathbb{E}(\cdot|M_{i_k})\ot Id_E\big)(x_l).
$$
Therefore, for any $1 \leq k \leq n$, we have
\begin{align*}
d_{2k}
   &=y_{2k}-y_{2k-1}\\
   &=\sum_{l=1}^{k}\varepsilon_l \ot \big(\mathbb{E}(\cdot|M_{i_k})\ot Id_E\big)(x_l)-\sum_{k=1}^{k-1}\varepsilon_l \ot \big(\mathbb{E}(\cdot|M_{i_k})\ot Id_E\big)(x_l)\\
   &=\varepsilon_k \ot \big(\mathbb{E}(\cdot|M_{i_k})\ot Id_E\big)(x_k)
\end{align*}
and similarly $d_{2k-1}=0$ for any integer $1 \leq k \leq n$.
Finally, we obtain that
\begin{align*}
\Bgnorm{\sum_{k=1}^{n} \varepsilon_k \ot
\big(\mathbb{E}(\cdot|M_{i_k})\ot
Id_E\big)(x_k)}_{L^p(L^\infty(\Omega) \ot M,E)}
    & = \Bgnorm{\sum_{k=1}^{n} d_{2k}}_{L^p(L^\infty(\Omega) \ot M,E)}\\
    & \lesssim \Bgnorm{\sum_{m=1}^{2n} d_m}_{L^p(L^\infty(\Omega) \ot M,E)}\\
    & =\norm{y_{2m}}_{L^p(L^\infty(\Omega) \ot M,E)}\\\\
    & = \bnorm{\big(\mathbb{E}(\cdot|N_{2n})\ot Id_E\big)(y)}_{L^p(L^\infty(\Omega) \ot M,E)}\\
    & \lesssim \norm{y}_{L^p(L^\infty(\Omega) \ot M,E)}=\Bgnorm{\sum_{k=1}^{n} \varepsilon_k \ot x_k}_{L^p(L^\infty(\Omega) \ot M,E)}.
\end{align*}
By using Khintchine-Kahane inequalities (e.g. \cite[page 211]{DJT}),
we conclude that
$$
\Bgnorm{\sum_{k=1}^{n} \varepsilon_k \ot
\big(\mathbb{E}(\cdot|M_{i_k}) \ot
Id_E\big)(x_k)}_{\Rad(L^p(M,E))}\lesssim \Bgnorm{\sum_{k=1}^{n}
\varepsilon_k \ot x_k}_{\Rad(L^p(M,E))}.
$$
We deduce the general case of von Neumann algebras equipped with
faithful normal semifinite traces by a straightforward application
of the well-known reduction method of Haagerup \cite{HJX}.
\end{proof}

Let $X$ be a Banach space. A strongly continuous semigroup
$(T_t)_{t\geq 0}$ is called $R$-analytic if there exist
$0<\theta<\frac{\pi}{2}$ and a holomorphic extension
$$
\begin{array}{cccc}
    &  \Sigma_\theta   &  \longrightarrow   &  B(X)  \\
    &  z   &  \longmapsto       &  T_z  \\
\end{array}
$$
with $R\big(\{T(z): z \in \Sigma_{\theta}\}\big) < \infty$. See
\cite{KW} for more information and for applications to maximal
regularity.

The following result is a particular case of \cite[Theorem
6.1]{Fac}.
\begin{thm}
\label{Th Fackler R-analyticity} Let $(T_{1,t})_{t \geq 0}$ and
$(T_{2,t})_{t \geq 0}$ be two consistent semigroups given on an
interpolation couple $(X_1, X_2)$ of $K$-convex Banach spaces.
Suppose $0<\theta<1$. Assume that $(T_{1,t})_{t \ge 0}$ is strongly
continuous and $R$-analytic and that $R\big(\{T_{2,t}: 0 < t <
1\}\big) < \infty$. Then there exists a unique strongly continuous
$R$-analytic semigroup $(T_t)_{t \geq 0}$ on $(X_1,X_2)_\theta$
which is consistent with $(T_{1,t})_{t \geq 0}$ and $(T_{2,t})_{t
\geq 0}$.
\end{thm}

Now, we prove Theorem \ref{Th R-Analyticity general semigroups}.

\begin{proof}[of Theorem \ref{Th R-Analyticity general semigroups}]
We only prove the hyper-factorizable case. The QWEP-factorizable
case is similar. We can identify $OH(I)$ with $\ell^2_I$ completely
isometrically. Fubini's theorem gives us an isometric isomorphism
$$
L^2\big(M,OH(I)\big)=\ell^2_I\big(L^2(M)\big).
$$
Hence the Banach space $L^2\big(M,OH(I)\big)$ is isometric to a
Hilbert space. On a Hilbert space, recall that any bounded set is
$R$-bounded. By Theorem \ref{Th Analyticity general semigroups}, we
deduce that $(T_t \ot Id_{OH(I)})_{t\geq 0}$ defines a $R$-analytic
semigroup on $L^p\big(M,OH(I)\big)$.

Let $t_1,\ldots, t_n$ be rational numbers. We take a common
denominator: we can write $t_i=\frac{s_i}{d}$ for some integers
$d,s_1,\ldots,s_n$. The operator $T_{\frac{1}{d}}$ admits a Rota
dilation:
$$
T_{\frac{1}{d}}^k=Q\mathbb{E}_k\pi,\qquad k \geq 1.
$$
Let $x_1,\ldots,x_n \in F$. Using Proposition \ref{Prop
noncommutative Stein on conditional expectations} and the fact that
any singleton is $R$-bounded, we obtain
\begin{align*}
\Bgnorm{\sum_{i=1}^{n}\varepsilon_i \ot (T_{t_i} \ot
Id_{F})(x_i)}_{\Rad(L^p(M,F))}
    &= \Bgnorm{\sum_{i=1}^{n}\varepsilon_i \ot (Q \mathbb{E}_{s_i}\pi\ot Id_{F})(x_i)}_{\Rad(L^p(M,F))}\\
    &\leq \Bgnorm{\sum_{i=1}^{n}\varepsilon_i \ot (\mathbb{E}_{s_i}\pi \ot Id_{F})(x_i)}_{\Rad(L^p(N,F))}\\
    &\leq R\big(\big\{\mathbb{E}_{k}\ot Id_{F}\ \colon \ k\geq 1\big\}\big) \Bgnorm{\sum_{i=1}^{n} \varepsilon_i \ot (\pi \ot Id_{F})x_i}_{\Rad(L^p(M,F))}\\
    &\leq R\big(\big\{\mathbb{E}_{k}\ot Id_{F}\ \colon \ k\geq 1\big\}\big) \Bgnorm{\sum_{i=1}^{n} \varepsilon_i \ot x_i}_{\Rad(L^p(M,F))}.
\end{align*}
By using the strong continuity of the semigroup, We conclude that
$\{T_t \ot Id_{F}\ :\ t \geq 0\}$ is a $R$-bounded subset of
$B\big(L^p(M,F)\big)$.

Now, we have
$$
L^p(M,E) = \Big( L^2\big(M,OH(I)\big),L^{q}(M,F)\Big)_{\alpha}.
$$
Furthermore, the operator space $F$ is $\OUMD_p$. We deduce that the
Banach space $S^p(F)$ is $\UMD$, hence K-convex. Thus the operator
space $F$ is $\OK$-convex. By Proposition \ref{Prop K-convexity of
Lp(M,E)}, we deduce that $L^{q}(M,F)$ is $\K$-convex. Moreover the
Banach space $L^2\big(M,OH(I)\big)$ is obviously $\K$-convex. We
conclude by using Theorem \ref{Th Fackler R-analyticity}.
\end{proof}

By example, we can take $E=OH(I)$ in Theorem \ref{Th R-Analyticity
general semigroups}.
\section{Applications to functional calculus}

We start with a little background on sectoriality and $H^\infty$ functional calculus. We refer to \cite{Haa}, \cite{KW}, \cite{JMX} and \cite{LM} for details and complements. A closed, densely defined linear operator $A \co D(A)\subset X \to X$ is called sectorial of type $\omega$ if its spectrum $\sigma(A)$ is included in the closed sector $\overline{\Sigma_\omega}$, and for any angle $\omega<\theta < \pi$, there is a positive constant $K_\theta$ such that
\begin{equation*}\label{Sector}
 \bnorm{(\lambda-A)^{-1}}_{X\to X}\leq
\frac{K_\theta}{\vert  \lambda \vert},\qquad \lambda
\in\mathbb{C}-\overline{\Sigma_\theta}.
\end{equation*}
We recall that sectorial operators of type $<\frac{\pi}{2}$ coincide
with negative generators of bounded analytic semigroups.

For any $0<\theta<\pi$, let $H^{\infty}(\Sigma_\theta)$ be the
algebra of all bounded analytic functions $f\colon \Sigma_\theta\to
\C$, equipped with the supremum norm $\norm{f}_{H^{\infty}(\Sigma_\theta)}=\,\sup\bigl\{\vert f(z)\vert \, :\, z\in \Sigma_\theta\bigr\}.$ Let
$H^{\infty}_{0}(\Sigma_\theta)\subset H^{\infty}(\Sigma_\theta)$ be
the subalgebra of bounded analytic functions $f\colon
\Sigma_\theta\to \C$ for which there exist $s,c>0$ such that $\vert
f(z)\vert\leq c \vert z \vert^s\bigl(1+\vert z \vert)^{-2s}\,$ for
any $z\in \Sigma_\theta$. Given a sectorial operator $A$ of type $0<
\omega < \pi$, a bigger angle $\omega<\theta<\pi$, and a function
$f\in H^{\infty}_{0}(\Sigma_\theta)$, one may define a bounded
operator $f(A)$ by means of a Cauchy integral (see e.g.
\cite[Section 2.3]{Haa} or \cite[Section 9]{KW}); the resulting
mapping $H^{\infty}_{0}(\Sigma_\theta)\to B(X)$ taking $f$ to $f(A)$
is an algebra homomorphism. By definition, $A$ has a bounded
$H^{\infty}(\Sigma_\theta)$ functional calculus provided that this
homomorphism is bounded, that is, there exists a positive constant
$C$ such that $\bnorm{f(A)}_{X \to X}\leq C\norm{f}_{H^{\infty}(\Sigma_\theta)}$ for any $f\in H^{\infty}_{0}(\Sigma_\theta)$. In the case when $A$ has a dense range, the latter boundedness condition allows a natural extension of $f\mapsto f(A)$ to the full algebra $H^{\infty}(\Sigma_\theta)$.

Suppose $1<p<\infty$. In the sequel, we say that an operator $A$ on
a vector valued noncommutative $L^p$-space $L^p(M,E)$ admits a
completely bounded $H^{\infty}(\Sigma_\theta)$ functional calculus
if $Id_{S^p} \ot A$ admits a bounded $H^{\infty}(\Sigma_\theta)$
functional calculus on the Banach space $S^p\big(L^p(M,E)\big)$.


We will use the following theorem \cite[Corollary 10.9]{KW}:

\begin{thm}
\label{Th Calcul fonct et dilatations} Let $(T_t)_{t \geq 0}$ be a
strongly continuous semigroup with generator $-A$ on a Banach space.
If the semigroup has a dilation to a bounded strongly continuous
group on a $\UMD$ Banach space, then $A$ has a bounded $H^{\infty}(\Sigma_\theta)$ functional calculus for some $0<\theta<\pi$.
\end{thm}
We also need the particular case of \cite[Proposition 5.1]{KaW}
which says that the presence of $R$-analyticity allows to reduce the
angle of a bounded $H^{\infty}(\Sigma_\theta)$ functional calculus
below $\frac{\pi}{2}$.

\begin{prop}
\label{Prop KW reduce angle} Let $(T_t)_{t \geq 0}$ be a strongly
continuous semigroup with generator $-A$ on a Banach space. Suppose
that $A$ has a bounded $H^{\infty}(\Sigma_{\theta_0})$ functional
calculus for some $0<\theta_0<\pi$. If the semigroup $(T_t)_{t \geq
0}$ is $R$-analytic then $A$ has indeed a bounded $H^{\infty}(\Sigma_\theta)$ functional calculus for some $0<\theta<\frac{\pi}{2}$.
\end{prop}

Now, we give the proof of our main result.

\begin{proof}[of Theorem \ref{Th calcul fonctionnel complètement borné}]
By definition, there exists a dilation of the semigroup $(T_t)_{t\geq
0}$. This dilation induces a dilation of the semigroup $(T_t \ot
Id_{E})_{t \geq 0}$ on $L^p(M,E)$ to a strongly continuous group of
isometries on $L^p(N,E)$.

Since the operator Hilbert space $OH(I)$ has $\OUMD_2$ and the operator space $F$ have $\OUMD_q$, we see by interpolation that the operator space $E$ has $\OUMD_p$. We infer that the Banach space $S^p(E)$ is $\UMD$. Hence the Banach space $S^p_I(E)$ is also $\UMD$ for any index set $I$. Since $N$ has QWEP, the Banach space $L^p(N,E)$ is a closed subspace of an ultrapower $S_I^p(E)^\ul$ of a vector valued Schatten space $S_I^p(E)$ for some index set $I$. We deduce that the Banach space $L^p(N,E)$ is $\UMD$.

By Theorem \ref{Th Calcul fonct et dilatations}, we deduce that
$A_p$ has a bounded $H^{\infty}(\Sigma_\theta)$ functional calculus
for some $0<\theta<\pi$. By Theorem \ref{Th R-Analyticity general
semigroups}, the semigroup $(T_t\ot Id_{E})_{t \geq 0}$ is $R$-analytic on $L^p(M,E)$. By Proposition \ref{Prop KW reduce angle}, these two results imply that $A_p$ actually admits a bounded
$H^{\infty}(\Sigma_\theta)$ functional calculus for some $0< \theta <\frac{\pi}{2}$, as expected.

Finally, it is not hard to check that the above arguments work as
well with $(Id_{B(\ell^2)} \ot T_t)_{t\geq 0}$ in the place of
$(T_t)_{t\geq 0}$. thus $A_p$ actually has a completely bounded
functional calculus for some $0<\theta<\frac{\pi}{2}$.
\end{proof}

Now, we give some natural examples
to which the results of this paper can be applied.
\paragraph{Noncommutative Poisson semigroup.}

Let $n\geq 1$ be an integer. Here $\mathbb{F}_n$ denotes a free
group with $n$ generators denoted by $g_1,\ldots,g_n$. Any $g\in
\mathbb{F}_n$ has a unique decomposition of the form
$$
g=g_{i_1}^{k_1}g_{i_2}^{k_2}\cdots g_{i_l}^{k_l},
$$
where $l \geq 0$ is an integer, each $i_j$ belongs to
$\{1,\ldots,n\}$, each $k_j$ is a non zero integer, and $i_j \not=
i_{j+1}$ if $1\leq j \leq l-1$. The case when $l=0$ corresponds to
the unit element $g=e_{\mathbb{F}_n}$. By definition, the length of
$g$ is defined as
$$
|g|=|k_1|+\cdots+|k_l|.
$$
This is the number of factors in the above decomposition of $g$. For
any nonnegative real number $t \geq 0$, we have a normal unital
completely positive selfadjoint map
$$
\begin{array}{cccc}
   T_t :&   \VN(\mathbb{F}_n)  &  \longrightarrow   &  \VN(\mathbb{F}_n)  \\
    &   \lambda_g  &  \longmapsto       &  e^{-t|g|}\lambda_g.  \\
\end{array}
$$
These maps define a w*-semigroup $(T_t)_{t \geq 0}$ called the
noncommutative Poisson semigroup (see \cite{JMX} for more
information). In \cite[remark following Proposition 5.5]{Arh}, it is
implicitly said that  $(T_t)_{t \geq 0}$ is $\QWEP$-dilatable.
Moreover, using \cite{Ric} and \cite[Proposition 4.8]{Arh}, we can
show each $T_t$ is $\QWEP$-factorizable.

\paragraph{$q$-Ornstein-Uhlenbeck semigroup.}

We use the notations of \cite{BKS}. Suppose $-1 \leq q <1$. Let $H$
be a real Hilbert space and let $(a_t)_{t \geq 0}$ be a strongly
continuous semigroup of contractions on $H$. For any $t\geq 0$, let
$T_t=\Gamma_q(a_t)$. Then $(T_t)_{\geq 0}$ is a w*-semigroup of
normal unital completely positive maps preserving the trace on the
von Neumann algebra $\Gamma_q(H)$. 

In the case where $a_t=e^{-t}I_H$, the semigroup $(T_t)_{\geq 0}$ is
the so-called $q$-Ornstein-Uhlenbeck semigroup.

Using \cite{Ric}, \cite{Nou} and the result \cite[Theorem 10.11]{KW}
of dilation of strongly continuous semigroups of contractions on a
Hilbert space it is not hard to see that we obtain examples of
$QWEP$-dilatable semigroups of $\QWEP$-factorizable maps.

We pass to Schur multipliers.
\begin{thm}\label{Th calcul fonctionnel borné}
Let $(T_t)_{t\geq 0}$ be a $w^*$-semigroup self-adjoint contractive
Schur multipliers on $B\big(\ell^2_I\big)$. Suppose $1< p,q <\infty$
and $0<\alpha<1$. Let $E$ be an operator space such that
$E=\big(OH(I),F\big)_\alpha$ for some index set $I$ and for some
$\OUMD_{q}$-operator space $F$ such that
$\frac{1}{p}=\frac{1-\alpha}{2}+\frac{\alpha}{q}$. We let $-A_p$ be
the generator of the strongly continuous semigroup $(T_t\ot
Id_{E})_{t\geq 0}$ on $S^p_I(E)$. Then for some
$0<\theta<\frac{\pi}{2}$, the operator $A_p$ has a bounded
$H^{\infty}(\Sigma_\theta)$ functional calculus.
\end{thm}

\begin{proof}
Arguing as in the proof of \cite[Corollary 4.3]{Arh}, we may reduce
the general case to the unital and completely positive case. The
proof for semigroups of unital completely positive Schur
multipliers, using \cite[Proposition 5.5]{Arh} and \cite{Arh2}, is
similar to the one of Theorem \ref{Th calcul fonctionnel
complètement borné}.
\end{proof}

\begin{remark}
The results of this paper lead to properties of some square
functions, see \cite[section 7]{LM} for more information.
\end{remark}


\textbf{Acknowledgment}. The author would like to express his
gratitude to Christian Le Merdy for some useful advices. He also
thank Hun Hee Lee and Quanhua Xu for some valuable discussions on
Proposition \ref{prop-tensorisation of CP maps} and \'Eric Ricard
for an interesting conversation.


\footnotesize{ \n Laboratoire de Math\'ematiques, Universit\'e de
Franche-Comt\'e,
25030 Besan\c{c}on Cedex,  France\\
cedric.arhancet@univ-fcomte.fr\hskip.3cm
\end{document}